\newif\ifpictures
\picturestrue

\documentclass[12pt]{amsart}
\usepackage{amssymb}
\usepackage{amsmath}
\usepackage{amscd}
\usepackage[dvips]{graphicx}
\usepackage{book tabs}
\usepackage{float}
\usepackage{todonotes}
\usepackage[normalem]{ulem}
\usepackage{hyperref} 

\numberwithin{equation}{section}
\newtheorem{thm}{Theorem}
\newtheorem{prop}[thm]{Proposition}
\newtheorem{lemma}[thm]{Lemma}
\newtheorem{cor}[thm]{Corollary}

\theoremstyle{definition}

\newtheorem{example}[thm]{Example}
\newtheorem{remark1}[thm]{Remark}

\newtheorem{question1}[thm]{Question}

\numberwithin{thm}{section}

\newcounter{FNC}[page]
\def\newfootnote#1{{\addtocounter{FNC}{2}$^\fnsymbol{FNC}$%
     \let\thefootnote\relax\footnotetext{$^\fnsymbol{FNC}$#1}}}

\newcommand{\C}{\mathbb{C}}
\newcommand{\N}{\mathbb{N}}

\newcommand{\R}{\mathbb{R}}

\renewcommand{\P}{\mathbb{P}}

\newcommand{\Sph}{\mathbb{S}}

\newcommand{\compl}{\mathsf{c}}

\DeclareMathOperator{\cone}{cone}

\DeclareMathOperator{\inter}{int}

\DeclareMathOperator{\diag}{diag}

\DeclareMathOperator{\rec}{rec}

\DeclareMathOperator{\init}{in}
\DeclareMathOperator{\pos}{pos}

\let\Im\relax
\DeclareMathOperator{\Im}{Im}

\usepackage{xcolor}

\title[]{Hyperbolicity cones and imaginary projections}

\author{Thorsten J\"orgens and Thorsten Theobald}

\address{Goethe-Universit\"at, FB 12 -- Institut f\"ur Mathematik,
Postfach 11 19 32, D--60054 Frankfurt am Main, Germany}

\email{\{joergens,theobald\}@math.uni-frankfurt.de}

\date{\today}

\begin{document}

\begin{abstract}
	Recently, the authors and de Wolff introduced the imaginary projection of a polynomial $f\in\C[\mathbf{z}]$ as the projection of the variety of $f$ onto its imaginary part, $\mathcal{I}(f) \ = \ \{\Im(\mathbf{z}) \, : \, \mathbf{z} \in \mathcal{V}(f) \}$. 
	Since a polynomial $f$ is stable if and only if $\mathcal{I}(f) \cap \R_{>0}^n \ = \ \emptyset$, the notion offers a 	novel geometric view underlying stability questions of polynomials.
		In this article, we study the relation between the imaginary projections and hyperbolicity cones, where the latter ones are only defined for homogeneous polynomials. Building upon this,  
for homogeneous polynomials we provide a tight upper bound for the number of components in the complement $\mathcal{I}(f)^{\compl}$ and thus for the number of hyperbolicity
cones of~$f$.
And we show that for $n \ge 2$, a polynomial $f$ in $n$ variables can have an arbitrarily high number of strictly convex and bounded components in $\mathcal{I}(f)^{\compl}$. 
\end{abstract}

\maketitle 

\section{Introduction}

A homogeneous polynomial $f\in\R[\mathbf{z}] =\R[z_1,\ldots,z_n]$ is called \emph{hyperbolic} in direction $\mathbf{e}\in\R^n$ if $f(\mathbf{e})\neq0$ and for every $\mathbf{x} \in \R^n$ the real function $t\mapsto f(\mathbf{x}+t\mathbf{e})$ has only real roots. 

We denote by $C(\mathbf{e})=\{\mathbf{x}\in\R^n:f(\mathbf{x}+t\mathbf{e})=0 \Rightarrow t<0\}$ the \emph{hyperbolicity cone} of $f$ with respect to $\mathbf{e}$. 
By G\r{a}rding's results \cite{garding-59}, 
$C(\mathbf{e})$ is convex, $f$ is hyperbolic with respect 
to every point $\mathbf{e}'$ in its hyperbolicity cone and $C(\mathbf{e})=C(\mathbf{e}')$ (see \cite{garding-59}). Note, that $0\notin C(\mathbf{e})$ and 
$-C(\mathbf{e})=C(-\mathbf{e})$ is a hyperbolicity cone of $f$ as well. Furthermore, hyperbolicity cones are open. Recent interest
in the hyperbolicity cones was supported by their application in hyperbolic 
programming (see \cite{gueler-97,nesterov-tuncel-2016,renegar-2006}) as well as by the open conjecture that every
hyperbolicity cone is spectrahedral (``Generalized Lax conjecture'', see 
\cite{vinnikov-2012} for an overview as well as \cite{amini-braenden-2015,HV2007,kpv-2015,LPR2005,sanyal-netzer-2015}).

In \cite{TTT}, the authors and de Wolff introduced the \emph{imaginary projection} of a polynomial
$f \in \C[\mathbf{z}]$ as the projection of the variety $\mathcal{V}(f)$ of $f$ onto its imaginary
part, $\mathcal{I}(f) \ = \ \{\Im(\mathbf{z}) \, : \, \mathbf{z} \in \mathcal{V}(f) \}$.
A polynomial $f\in\R[\mathbf{z}]$ is (real) \emph{stable}, i.e., its imaginary projection does not intersect the positive orthant, if and only if $f$ is hyperbolic with respect to every point in the positive orthant, see \cite{garding-59, wagner-2011}.
The complement of the closure of $\mathcal{I}(f)$ consists of finitely many convex components, thus
offering strong connections to the theory of amoebas (see \cite{TTT}).

The main goal of this paper is to study the number of complement components of the imaginary projection of a polynomial $f$.
In the homogeneous case, it turns out that this question is equivalent to characterizing the 
number of hyperbolicity cones of $f$:

\begin{thm}\label{thm:connection-complement-hyperbolicity-cone}
	Let $f\in\R[\mathbf{z}]$ be homogeneous. Then the hyperbolicity cones of $f$ coincide with the components of $\mathcal{I}(f)^{\compl}$.
\end{thm}

Hence,
imaginary projections offer a geometric view on the collection of all
hyperbolicity cones of a given polynomial. Building upon this,
we can provide the following sharp upper bound for homogeneous polynomials.

\begin{thm}\label{th:thm1}
	Let $f \in \R[\mathbf{z}]$ be homogeneous of degree $d$. 
	Then the
   number of hyperbolicity cones of $f$ and thus the number of components in the complement of $\mathcal{I}(f)$ is at most 
\[
  \begin{cases}
     2^{d} & \text{ for } d\leq n \, , \\
     2 \sum_{k=0}^{n-1}\binom{d-1}{k} & \text{ for } d>n \, .
  \end{cases}
\]
 The maximum is attained if and only if $f$ is a product of independent linear polynomials in the sense that any $n$ of them are linearly independent.
\end{thm}

If a part of the boundary of a
complement component comes from a linear factor, then the complement
component is not strictly convex. 
It seems to be open whether for given dimension $n$, the number of 
strictly convex cones in the complement of $\mathcal{I}(f)$ can become
arbitrarily large for a homogeneous polynomial $f$.
We show in Theorem~\ref{th:strict-convex-nonhomog}, that a non-homogeneous polynomial 
$f$ can have an arbitrarily large number of \emph{strictly convex}, bounded 
components in the complement. 

\begin{thm}\label{th:strict-convex-nonhomog}
Let $n \ge 2$. For any $K > 0$ there exists a polynomial $f \in \R[\mathbf{z}]$
such that $f$ has at least $K$ strictly convex, bounded 
components in the complement of $\mathcal{I}(f)$.
\end{thm}

The following question remains open:

\begin{question1}
	Given a non-homogeneous
	polynomial $f \in \C[\mathbf{z}]$ of total degree $d$ (or with given Newton polytope $P$), how many bounded or unbounded 
	components can the complement of $\mathcal{I}(f)$ at most have?
\end{question1}

The paper is structured as follows. Section~\ref{se:homogeneous} deals with the
connections of imaginary projections of homogeneous polynomials and hyperbolicity cones
and proves Theorem~\ref{thm:connection-complement-hyperbolicity-cone}.
Section~\ref{se:proofs} contains
the proof of Theorem~\ref{th:thm1}, 
and it characterizes the boundary of imaginary projections.
Section~\ref{se:nonhomogeneous} then deals with inhomogeneous polynomials and proves 
Theorem~\ref{th:strict-convex-nonhomog}.

\section{Imaginary projection of homogeneous polynomials and hyperbolicity cones\label{se:homogeneous}}

Throughout the paper, we use bold letters for vectors, 
e.g., $\mathbf{z}=(z_1,\ldots,z_n)\in\C^n$. If not stated otherwise, 
the dimension is $n$. Denote by $\mathcal{V}(f)$ the complex variety of 
a polynomial $f$ and by $\mathcal{V}_\R(f)$ the real variety of $f$.
Moreover, set $A^{\compl} = \R^n \setminus A$ for the complement of
a set $A \subseteq \R^n$.

For the notion of hyperbolic polynomials, one usually starts from real homogeneous polynomials,
while imaginary projections can be defined also for non-homogeneous and for complex
polynomials. For coherence, 
we also consider the notion of a hyperbolic polynomial for
complex homogeneous polynomials $f \in \C[\mathbf{z}]$. Note that if $f \in \C[\mathbf{z}]$ 
is hyperbolic with respect to $\mathbf{a} \in \R^n$, then $f(\mathbf{z})/f(\mathbf{a})$ has real coefficients and is hyperbolic with respect to $\mathbf{a}$ as well (see \cite{garding-59}).

For homogeneous polynomials, we now prove the connection between hyperbolicity cones 
and imaginary projections stated in Theorem~\ref{thm:connection-complement-hyperbolicity-cone}. It generalizes the relation between homogeneous real stable polynomials and hyperbolicity cones, which was already mentioned in the introduction.
Note that for a homogeneous polynomial $f$, the imaginary
projection $\mathcal{I}(f)$ can be regarded as a (non-convex) cone, i.e.,
for any $\mathbf{z} \in \mathcal{I}(f)$ and $\lambda \ge 0$ we have $
\lambda \mathbf{z} \in \mathcal{I}(f)$.
Thus, in particular, $0\in\mathcal{I}(f)$.

\begin{proof}[Proof of Theorem \ref{thm:connection-complement-hyperbolicity-cone}]
	 We show the following two properties:
	 \begin{enumerate}
	 	\item If $f$ is hyperbolic with respect to $\mathbf{e}\in\R^n$, then the hyperbolicity cone $C(\mathbf{e})$ satisfies $C(\mathbf{e}) \subseteq \mathcal{I}(f)^{\compl}$.
	 	\item If there is a convex cone $C$ with $C\subseteq\mathcal{I}(f)^{\compl}$, then $f$ is hyperbolic with respect to every point in $C$, i.e., $C$ is contained in that hyperbolicity cone of $f$.
	 \end{enumerate}
	Assume first, that $f$ is hyperbolic with respect to $\mathbf{e}\in\R^n$, and let $\mathbf{e}'\in C(\mathbf{e})$. Then $\mathbf{e}'$ cannot be the imaginary part of a root $\mathbf{z}=\mathbf{x}+i\mathbf{y}$, since otherwise $i$ would be a non-real zero of the univariate function $t\mapsto f(\mathbf{x}+t\mathbf{e}')$.
	
	Assume now that there is a convex cone $C$ with $C\subseteq\mathcal{I}(f)^{\compl}$. The homogeneity of $f$ implies $-C\subseteq\mathcal{I}(f)^{\compl}$. For $\mathbf{e}\in \pm C$, we have $f(\mathbf{x}+i\mathbf{e})\neq 0$ for all $\mathbf{x}\in\R^n$, which gives in particular 
	\[
	f(\mathbf{e})=(1+i)^{-\deg f}f((1+i)\mathbf{e})=(1+i)^{-\deg f}f(\mathbf{e}+i\mathbf{e})\neq0,
	\]
where $\deg(f)$ denotes the degree of the homogeneous polynomial $f$.
	Furthermore, if there were an $\mathbf{x}\in\R^n$ such that $t\mapsto f(\mathbf{x}+t\mathbf{e})$ has a non-real solution $a+ib$, $b\neq0$, then 
	\[
	f(\mathbf{x}+a\mathbf{e}+ib\mathbf{e})=0
	\]
	in contradiction to $b\mathbf{e}\in\pm C\subseteq\mathcal{I}(f)^{\compl}$.
\end{proof}

As an immediate consequence, we obtain the following description of the imaginary projection of a homogeneous polynomial.
\begin{cor}\label{cor:hom-complement-two-cones}
	If $f\in\C[\mathbf{z}]$ is homogeneous, then its imaginary projection is a closed cone
       (in general non-convex).
	The components $C_1, \ldots, C_t$ of $\mathcal{I}(f)^{\compl}$ 
  are hyperbolicity cones of $f$ and occur pairwise,
       with $C_{i_1} = -C_{i_2}$.
	In particular, the imaginary projection of a homogeneous polynomial has no bounded components in its complement.
\end{cor}
\begin{proof}
By Theorem~\ref{thm:connection-complement-hyperbolicity-cone}, the components $C_1, \ldots, C_t$ of $\mathcal{I}(f)^{\compl}$ are the hyperbolicity cones of 
$f$, which occur pairwise.
Since hyperbolicity cones are open and since $\mathcal{I}(f)$ has only finitely many of these
conic components in the complement, $\mathcal{I}(f)$ is closed.
And since $\mathcal{I}(f)$ is a cone, there are no bounded components in the complement.
\end{proof}

The following examples illustrate the connection stated in Theorem \ref{thm:connection-complement-hyperbolicity-cone} in well-known cases.

\begin{example}
	Let $f(\mathbf{z}) =z_1\cdots z_n$. Then $f$ is hyperbolic with respect to every point $\mathbf{e}\in(\R\setminus\{0\})^n$. Setting $z_j = x_j + i y_j$,
we obtain
\[
  \mathcal{I}(f) \ = \ \{\mathbf{y} \in \R^n \, : \, \prod_{j=1}^n (x_j + iy_j) = 0 \text{ for some } \mathbf{x} \in \R^n \} \ = \ \bigcup_{j=1}^n\{\mathbf{y}\in\R^n:y_j=0\}.
\]
\end{example}

\begin{example}\label{ex:SOC}
	Let $f(\mathbf{z})=z_1^2-\sum_{j=2}^nz_j^2$, $n>2$. 
It is well-known that $f$ is hyperbolic with respect to 
	any point $\mathbf{e}\in\R^n$ with $e_1^2-\sum_{j=2}^ne_j^2>0$ (e.g., \cite[Example 1]{garding-59}), and that the two hyperbolicity cones are
the open \emph{second-order cone} (or open \emph{Lorentz cone})
$\mathcal{L} = \{ \mathbf{x} \in \R^n \, : \, x_1^2-\sum_{j=2}^nx_j^2>0, \, x_1 > 0\}$ and
its negative $- \mathcal{L}$.
Likewise, the imaginary projection of $f$ is $	\mathcal{I}(f)= \{ \mathbf{y} \in \R^n \, : \, y_1^2 - \sum_{j=2}^n y_j^2 \le 0\} = \R^n \setminus (\mathcal{L} \cup - \mathcal{L})$,
which was computed as part of \cite[Theorem 5.4]{TTT}. This illustrates
Theorem~\ref{thm:connection-complement-hyperbolicity-cone}.

Furthermore, if a real homogeneous quadratic polynomial $f \in \R[\mathbf{z}]$ is hyperbolic, then
its hyperbolicity cone is the image of the second-order cone under a linear transformation; that property follows from
the classification of the imaginary projections of real homogeneous quadratic polynomials
in \cite{TTT}.
\end{example}

\begin{example}\label{ex:spectrahedron}
	Let $f(\mathbf{z})=\det(z_1A_1+\dots+z_nA_n)$, where $A_1,\ldots,A_n$ are Hermitian $d\times d$-matrices. It is well-known \cite[Prop.~2]{LPR2005}, that $f$ has the spectrahedral hyperbolicity cone
	\[
	C \ = \{\mathbf{x}\in\R^n:x_1A_1+\dots+x_nA_n\succ0\}.
	\]
	This implies that $\pm C$ are components of $\mathcal{I}(f)^{\compl}$. We can compute directly that $\mathcal{I}(f)^{\compl}$ has exactly these two components. Namely, given some $\mathbf{y}\in\R^n$ with $A(\mathbf{y}):=y_1A_1+\dots+y_nA_n\succ0$, we have
	\begin{equation}\label{eq:example-determinant}
	f(\mathbf{x}+i\mathbf{y})=\det\left(\sum_{j=1}^n x_jA_j+iA(\mathbf{y})\right)=\det (A(\mathbf{y}))\cdot\det\left(\sum_{j=1}^n x_j A(\mathbf{y})^{-1/2}A_jA(\mathbf{y})^{-1/2}+iI\right) \, ,
	\end{equation}
  where $A(\mathbf{y})^{-1/2}$ is the unique matrix with $A(\mathbf{y})^{-1/2} \cdot A(\mathbf{y})^{-1/2} = A(\mathbf{y})^{-1}$.
	If $f(\mathbf{x}+i\mathbf{y})$ vanished for some $\mathbf{x}\in\R^n$, then the Hermitian matrix 
$\sum x_j A(\mathbf{y})^{-1/2}A_jA(\mathbf{y})^{-1/2}$
would have the eigenvalue $-i$. But this is impossible, since Hermitian matrices have only real eigenvalues. Hence, $\mathbf{y}\notin\mathcal{I}(f)$.
	
	Conversely, let $f(\mathbf{x}+i\mathbf{y})=0$. Assuming $A(\mathbf{y})\succ0$, the right hand side of~\eqref{eq:example-determinant} vanishes, which again gives the contradiction that $-i$ is an eigenvalue of the Hermitian matrix.
\end{example}

In the following, we consider the number and structure of hyperbolicity cones of a homogeneous polynomial. In order to see that there can appear many hyperbolicity cones, consider polynomials of the form 
$f(\mathbf{z}) = \det(A_1z_1+\dots+A_nz_n)$ with real diagonal $d\times d$-matrices. 
This is a special case of Example \ref{ex:spectrahedron}, where the spectrahedral hyperbolicity cone becomes a polyhedron. In that case, it becomes profitable to use the viewpoint of imaginary
projections to describe exactly the hyperbolicity cones of their complement. Namely,
$\mathcal{I}(f)$ is an algebraic variety here, whereas the hyperbolicity cones are semi-algebraic.

\begin{thm}\label{thm:det-diag-imag-proj}
	Let $f(\mathbf{z})=\det(A_1z_1+\dots+A_nz_n)$, where $A_1,\ldots,A_n$ are $d\times d$ real diagonal matrices, $A_j=\diag(a_1^{(j)},\ldots,a_n^{(j)})$. Then $\mathcal{I}(f)$ is the hyperplane arrangement 
	\begin{equation}\label{eq:det-diag-imag-proj}
	\mathcal{I}(f)=\bigcup_{l=1}^d\{\mathbf{y}\in\R^n:\sum_{j=1}^n a_{l}^{(j)}y_j=0\}.
	\end{equation}
\end{thm}

Lemma~\ref{th:number-hyperbolicity-cones-linear-polynomials} in Section~\ref{se:proofs} will show that if $d'$ is the number of distinct hyperplanes in \eqref{eq:det-diag-imag-proj}, then the number of complement components is at most $2^{d'}$ for $d'\leq n$ and at most $2\cdot\sum_{k=0}^{n-1}\binom{d'-1}{k}$ for $d'>n$.

\begin{proof}
	We have 
	\begin{equation}\label{eq:det-diag-imag-proj-proof}
	\det(A_1z_1+\dots+A_nz_n)=\prod_{l=1}^d\left(\sum_{j=1}^n a_l^{(j)}x_l+i\sum_{j=1}^n a_l^{(j)}y_l\right).
	\end{equation}
	Assume that there is some $\mathbf{y}\in\R^n$ such that $\sum_{j=1}^n a_{l}^{(j)}y_j=0$ for an $l\in\{1,\ldots,d\}$. Then, choosing $\mathbf{x}=\mathbf{y}$, we have $f(\mathbf{x}+i\mathbf{y})=0$.\\
	Assume now $\sum_{j=1}^n a_{l}^{(j)}y_j\neq0$ for all $1\leq l\leq d$. Since \eqref{eq:det-diag-imag-proj-proof} vanishes if and only if at least one factor vanishes, we have $f(\mathbf{x}+i\mathbf{y})\neq0$ for all $\mathbf{x}\in\R^n$.
\end{proof}

We conclude the section with an exact statement on the number of
hyperbolicity cones in the bivariate case.

\begin{thm}
	Let $f\in\R[z_1,z_2]$ be homogeneous and of degree $d$. Then $f$ has at most $2d$ hyperbolicity cones. The exact number depends on the 
number of distinct solutions of $f(1,z_2)=0$:
	\begin{enumerate}
		\item If there is at least one complex solution, then there are no hyperbolicity cones.
		\item If there are $k$ distinct real solutions, then there are $2k$ hyperbolicity cones.
	\end{enumerate}
\end{thm}

For the proof, recall the definition of the set of \emph{limit directions} as the set
of limit points of points in $\frac{1}{r}\mathcal{I}(f)\cap\Sph^{n-1}$ (for $r \to \infty$),
written $\mathcal{I}_\infty(f) \ = \ \lim_{r\rightarrow \infty}\left(\frac{1}{r}\mathcal{I}(f)\cap\Sph^{n-1}\right)$,
which describes the behavior at infinity of the imaginary projection of a polynomial $f\in\C[\mathbf{z}]$.
The following statement was shown in \cite[Cor. 6.7]{TTT}.

\begin{prop}\label{pr:limitdirections}
        Let $f\in\C[z_1,z_2]$ be of total degree $d$ and assume its homogenization $f_h \in \C[z_0,z_1,z_2]$ has the zeros at infinity $(0:1:a_j)$, $j=1,\ldots,d$. Then
\[
  \mathcal{I}_\infty(f) \ = \ \begin{cases}
     \bigcup\limits_{j=1}^d \left\{\pm\frac{1}{\sqrt{1+a_j^2}}(1,a_j)\right\} &
        \text{if all $a_j$ are real,}\\
          \Sph^1 & \text{otherwise.}
   \end{cases}
\]
\end{prop}

\begin{proof}[Proof of Theorem~2.6.]
By Proposition~\ref{pr:limitdirections}, either every point on $\Sph^1$ is a limit
direction of $f$, or $f$ has at most $2d$ limit directions. Since
$f$ is homogeneous, we have $\mathcal{I}_{\infty}(f) = \mathcal{I}(f) \cap \Sph^{n}$.
Hence, every connected component of the complement of 
$\mathcal{I}(f)\cap\Sph^1$ on the sphere
corresponds to a hyperbolicity cone of $f$.
The more precise characterization then follows from the more refined 
characterization in Proposition~\ref{pr:limitdirections} as well.
\end{proof}

\section{Proof of Theorem~\ref{th:thm1}\label{se:proofs}}

In this section, we prove Theorem~\ref{th:thm1} on the maximal number of hyperbolicity cones of homogeneous polynomials.
Moreover, as a consequence of the results on the hyperbolicity cones, we provide a characterization of the boundary of the imaginary projections of homogeneous polynomials in Theorem~\ref{th:boundaryhomog}.

For the maximal number of hyperbolicity cones,
it will turn out that this number is achieved by polynomials which are
products of independent linear factors.

\begin{lemma}\label{th:number-hyperbolicity-cones-linear-polynomials}
	Let $f(\mathbf{z})=p_1(\mathbf{z})\cdots p_d(\mathbf{z}) \in 
\C[\mathbf{z}]$ 
be a product of $d$ linear polynomials $p_1,\ldots,p_d$.
Unless $\mathcal{I}(f) = \R^n$, 
the number of hyperbolicity cones of $f$ is positive and at most
		\begin{enumerate}
			\item $2^{d}$ for $1\leq d\leq n,$
			\item $2 \sum_{k=0}^{n-1}\binom{d-1}{k}$ for $d>n$.
		\end{enumerate}
\end{lemma}

Before the proof, we recall the following statement on linear polynomials
from~\cite{TTT}, phrased there in the affine setting.

\begin{prop}\label{prop:linear-poly}
For every homogeneous linear polynomial
$f(\mathbf{z}) = \sum_{j=1}^n a_j z_j \in \R[\mathbf{z}]$
with $(a_1, \ldots, a_n) \neq 0$, we have
$\mathcal{I}(f) = \mathcal{V}_{\R}(\sum_{j=1}^n a_j y_j)$.
If the coefficients of $f$ are complex, then $\mathcal{I}(f)$ is either
a hyperplane or $\mathcal{I}(f) = \R^n$.
\end{prop}

\medskip

\noindent
\emph{Proof of Theorem \ref{th:number-hyperbolicity-cones-linear-polynomials}.}
Let $f(\mathbf{z})=p_1(\mathbf{z})\cdots p_d(\mathbf{z})$ be a product of 
$d$ linear polynomials $p_1,\ldots,p_d$ and $\mathcal{I}(f) \neq \R^n$.
Since $\mathcal{I}(p_j)$ is a hyperplane for all $j$, the imaginary projection $\mathcal{I}(f)$
defines a central hyperplane arrangement in $\R^n$, where central expresses that
all the hyperplanes are passing through the origin. We can assume that the hyperplanes are in 
general position, since otherwise the number of hyperbolicity cones may only become smaller.

By Zaslavsky's results \cite{zaslavsky-1975} (see also \cite[Prop.~2.4]{stanley-hyperplane}),
the number of chambers in an affine hyperplane arrangement of $d$
affine hyperplanes in general position is $\sum_{k=0}^{n} \binom{d}{k}$,
out of which $\binom{d-1}{n}$ chambers are bounded. Determining the number of chambers
in a central hyperplane arrangement of $d$ affine hyperplanes in general
position can be reduced to an affine hyperplane arrangement in
$\R^{n-1}$ and gives
\[
  \sum_{k=0}^{n-1} \binom{d}{k} + \binom{d-1}{n-1} \ = \ 
  2 \sum_{k=0}^{n-1} \binom{d-1}{k} \, .
\]
For $1 \le d \le n$, by the Binomial Formula this specializes to the 
expression given.
\hfill $\Box$

\medskip

By the results of Helton and Vinnikov \cite{HV2007},
the real variety of a smooth and hyperbolic polynomial consists of nested ovals 
(and a pseudo-line in case of odd degree) in the projective space $\P^{n-1}$.
Hence, the hyperbolicity cone is unique (up to sign). Motivated by an earlier 
version of the present article, Kummer was able to weaken the precondition
and showed that even for irreducible hyperbolic polynomials the hyperbolicity 
cone is unique (up to sign). 

\begin{prop}[\cite{kummer2017}]\label{th:thm2}
	Let $f \in \R[\mathbf{z}]$ be an irreducible homogeneous polynomial. 
  Then $f$ has at most two hyperbolicity cones
       (i.e., one pair)
      and thus at most two components in the complement of $\mathcal{I}(f)$.
\end{prop}

\begin{lemma}\label{lemma:hyperbolicity-cones-product-smooth-linear}
	Let $f_1, f_2 \in \C[\mathbf{z}]$ be homogeneous and $f_1$ be irreducible. 
Then the number of hyperbolicity cones of $f_1 \cdot f_2$ is at most twice the
number of hyperbolicity cones of $f_2$.
\end{lemma}
\begin{proof}
First note that any hyperbolicity cone $C$ of $f_1 \cdot f_2$ is of the form
$C = C_1 \cap C_2$ with hyperbolicity cones $C_1$ and $C_2$ of $f_1$ and $f_2$.

We can assume that $f_1$ and $f_2$ are hyperbolic. Then, by Theorem \ref{th:thm2}, $f_1$ has at most one pair of hyperbolicity cones. Intersecting these two cones with the hyperbolicity cones of $f_2$ gives the bound.
\end{proof}

Since the lemma inductively extends to an arbitrary number of factors, two or more pairs of hyperbolicity cones only arise from different factors in the polynomial $f$. This fact is captured explicitly by Theorem \ref{th:thm1}, whose proof is now given. 

\begin{proof}[Proof of Theorem \ref{th:thm1}]
Since the case $n=1$ is trivial, we can assume $n \ge 2$. 
Let $f=p_1\cdots p_k$ be a homogeneous polynomial of degree $d$, where
$p_1, \ldots, p_k$ are irreducible. Hence, $d = \deg(p_1)+\cdots+\deg(p_k)$.
We construct a polynomial $g = q_1 \cdots q_k$ with linear polynomials $q_i$ such
that $g$ has at least as many hyperbolicity cones as $f$.

By Lemma~\ref{lemma:hyperbolicity-cones-product-smooth-linear}, the number
of hyperbolicity cones of $f$ is at most twice the number of hyperbolicity
cones of $p_2 \cdots p_k$. 
Since the irreducible polynomial $p_1$ has at most two hyperbolicity cones, there exists
some hyperplane $H$ separating these two (open) convex cones. Set $q_1$ to be a linear
polynomial whose zero set is $H$. The set of hyperbolicity cones of $f$
injects to the set of hyperbolicity cones of $f^* = q_1 p_2 \cdots p_k$.
Repeating this process for $p_2, \ldots, p_k$ provides a polynomial $g=q_1 \cdots q_k$
whose number of hyperbolicity cones is at least the number of hyperbolicity
cones of $f$.

Hence, the number of hyperbolicity cones is maximized if $f$ is a product
of independent linear polynomials. Since replacing any nonlinear polynomial $p_i$ by a linear 
polynomial $q_i$ decreases the total degree of the overall product, 
the maximum number of hyperbolicity cones of a degree $d$ polynomial
cannot be attained if $f$ has a nonlinear irreducible factor $p_i$.

Now the stated numbers follow from 
Lemma~\ref{th:number-hyperbolicity-cones-linear-polynomials}.
\end{proof}

An illustration, where this number is attained, is given by Theorem \ref{thm:det-diag-imag-proj}.

For homogeneous polynomials $f$,
the uniqueness statement~Proposition~\ref{th:thm2} (up to sign) 
allows to characterize the boundary of $\mathcal{I}(f)^{\compl}$ -- or equivalently the boundary of the
hyperbolicity cones -- in terms of the variety $\mathcal{V}(f)$.

\begin{thm}\label{th:boundaryhomog}
	Let $f \in \C[\mathbf{z}]$ be homogeneous. Then
	\begin{enumerate}
		\item $\mathcal{V}_\R(f)\subseteq\mathcal{I}(f)$, with equality if and only if $e^{i\phi} f$ 
		is a product of real linear polynomials for some $\phi \in [0,2 \pi)$.
		\item If $f$ is hyperbolic and irreducible, then the Zariski closure of the boundary of $\mathcal{I}(f)^{\compl}$ equals $\mathcal{V}(f)$,
		\[
		\overline{\partial\mathcal{I}(f)^{\compl}}^Z=\mathcal{V}(f).
		\]
	\end{enumerate}
\end{thm}

\begin{proof}
By homogeneity, if $\mathbf{z}$ is a root of $f$, then $i\mathbf{z}$ is a root of $f$ as well. Hence, if $\mathbf{x}\in\mathcal{V}_\R(f)$, then $\mathbf{x}\in\mathcal{I}(f)$.

	Let $e^{i\phi}f$ be a linear polynomial with real coefficients. 
	By Proposition~\ref{prop:linear-poly}, the imaginary projection of $e^{i \phi}f$ and thus of
	$f$ is exactly $\mathcal{V}_\R(f)$ (notice that $\mathcal{I}(f)\neq\R^n$). Hence, the statement holds for products of linear polynomials as well.

	For the converse direction, let $\mathcal{I}(f)=\mathcal{V}_\R(f)$. Assume first that $f$ is irreducible. We observe that $f$ must be hyperbolic, since otherwise $\mathcal{I}(f)=\R^n$, which would imply $f\equiv0$.  By Proposition \ref{th:thm2}, $f$ has exactly one pair of hyperbolicity cones. It corresponds to the two convex, open components $C$ and $-C$ of $\mathcal{I}(f)^{\compl}$. By assumption, $\mathcal{I}(f)$ is a real algebraic set, and hence
	$\mathcal{I}(f)=\partial\mathcal{I}(f) = \partial C = \partial (-C)$.
	Thus, $\mathcal{I}(f)=\overline{C}\cap\overline{-C}$ is a convex set, where $\overline{C}$ denotes the topological closure of $C$. 
	Since for any two points $\mathbf{a}, \mathbf{b}\in\mathcal{I}(f)$ with $\mathbf{a}\neq \mathbf{b}$ their convex combination is contained in $\mathcal{I}(f) = \mathcal{V}_\R(f)$, and hence the underlying polynomial must be linear.
	Due to $\mathcal{I}(f) \neq \R^n$, the classification of linear polynomials in 
	\cite{TTT} (cf.\ Prop.~\ref{prop:linear-poly} here) provides that
	$f$ is of the form $e^{i \phi} f$.
	
	If $f$ is a product of non-constant irreducible polynomials, we can consider the imaginary projection of each factor and obtain the overall statement.

For the second statement, let $f$ be hyperbolic with respect to $\mathbf{e}$
and irreducible. By Theorem \ref{thm:connection-complement-hyperbolicity-cone},
the hyperbolicity cone $C = C(\mathbf{e})$ is a component of 
$\mathcal{I}(f)^{\compl}$. 
Since $C$ is the connected component
in the complement of $\mathcal{V}(f)$ containing $\mathbf{e}$ (see \cite{renegar-2006}), it is bounded by some subset of its real variety. And since $f$ is irreducible, the Zariski closure of $\partial\mathcal{I}(f)^{\compl}$ is $\mathcal{V}(f)$.
\end{proof}

\section{Non-homogeneous polynomials and their homogenization\label{se:nonhomogeneous}}

In this section, we deal with the complement components for non-homogeneous polynomials
as well as with homogenization. For $f \in \C[\mathbf{z}]$, we show that there is a bijection between 
the set of unbounded components of $\mathcal{I}(f)^{\compl}$ with full-dimensional
recession cone and the hyperbolicity cones of the initial form of $f$ (as defined below). Then we show Theorem~\ref{th:strict-convex-nonhomog}.

Denote by $f_h=f_h(z_0,\mathbf{z})$ the homogenization of $f$ with respect to the variable $z_0$.
For a set $X \subseteq \R^n$ let $\cone X = \{ \lambda \mathbf{x} \in \R^n \, : \, \mathbf{x} \in X, \lambda \ge 0\}$ denote the \emph{cone over} $X$.
The following statement captures the connection between the imaginary projection of $f$
and the imaginary projection of its homogenization.

\begin{thm} \label{thm:relation-imag-homog}
If $f\in\C[\mathbf{z}]$ then
	$
	\mathcal{I}(f_h)\cap\{(y_0,\mathbf{y}) \in \R^{n+1} \,: \, y_0=0\} 
\ = \ \{0\} \times ( \cone\mathcal{I}(f) \cup
\mathcal{I}(f_h(0,\mathbf{z}))).
	$
\end{thm}
\begin{proof}
	If $\mathbf{y}$ is a non-zero point in $\cone \mathcal{I}(f)$, we have $\lambda\mathbf{y}\in\mathcal{I}(f)$ for some 
$\lambda \ge 0 $. Hence, there exists an $\mathbf{x} \in \R^n$ with
$f_h((1,\mathbf{x}+i\mathbf{\lambda y})) = 0$. 
By homogeneity of $f_h$, this also gives
$(0, \mathbf{y}) \in \mathcal{I}(f_h) \cap \{(y_0,\mathbf{y}) \in \R^{n+1} \, : \, y_0 = 0\}$.

Conversely, if $(0,\mathbf{y})$ is a non-zero point in 
$\mathcal{I}(f_h)\cap\{(y_0,\mathbf{y}) \in \R^{n+1} \, : \, y_0=0\}$
 and $y \not\in \mathcal{I}(f_h(0,\mathbf{z}))$,
 then there exists some 
$\mathbf{x} \in \R^n$ and some $c \in \R \setminus \{0\}$ such that
$f_h((c,\mathbf{x} + i\mathbf{y})) = 0$. Hence, $\frac{1}{c}(\mathbf{x}+i\mathbf{y})$
is a zero of $f$, and therefore $\mathbf{y} \in \cone \mathcal{I}(f)$.
\end{proof}

By Theorem~\ref{thm:relation-imag-homog},
bounded components in the complement vanish under homogenization, and 
only conic components with apex at the origin remain.
Concerning dehomogenization, note that the intersection of the imaginary projection
of a homogeneous polynomial $f \in \C[z_0,\mathbf{z}] = \C[z_0, \ldots, z_n]$ with a fixed hyperplane  
$\{(y_0,\mathbf{y}) \in\R^{n+1} \, : \, y_0=\beta\}$, $\beta\neq0$ is
\[
	\mathcal{I}(f_h)\cap\{(y_0,\mathbf{y})\in\R^{n+1} \, : \, y_0=\beta\}=\bigcup_{\alpha\in\R}\mathcal{I}(f_h(\alpha+i\beta,\mathbf{z})).
\]

We denote by $\init(f)$ the \emph{initial form} of $f$, i.e., the sum of all those terms which have maximal total degree. Note that
$\init(f)(\mathbf{z})=f_h(0,\mathbf{z})$. 

Recall that the \emph{recession cone} of a convex set $A \subseteq \R^n$ is
$\rec(A) \, = \, \{ \, \mathbf{a} \in A \, : \, \mathbf{a} + \mathbf{x} \in A \text{ for all $\mathbf{x}$ in $A$}\}$ 
(see, e.g., \cite{rockafellar-book}). Whenever $A$ is closed then $\rec(A)$
is closed. For a polynomial $f$, 
denoting by $\overline{\mathcal{I}(f)}$ the closure of $\mathcal{I}(f)$,
we can characterize the components of
$(\overline{\mathcal{I}(f)})^{\compl}$ with full-dimensional recession cones
in terms of the hyperbolicity cones of $\init(f)$.

\begin{thm}\label{thm:unbounded-complement-components-dim-n}
  For $f \in \C[\mathbf{z}]$,
  there is a bijection between the set of unbounded components of 
  $\mathcal{I}(f)^{\compl}$ with full-dimensional recession cone and the hyperbolicity cones of $\init(f)$.
\end{thm}

\begin{figure}
	\centering
	\includegraphics[width=0.3\linewidth]{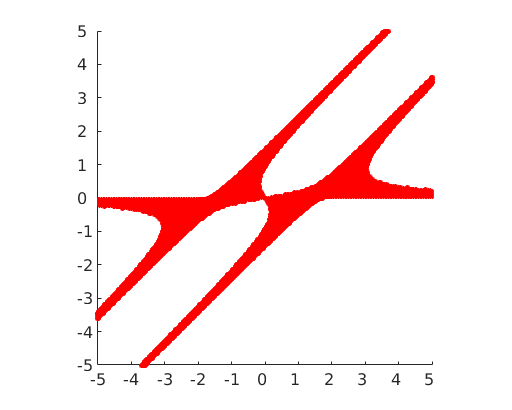}
	\caption{$f(z_1,z_2)=z_1^3-2z_1^2z_2+z_1z_2^2+z_1+z_2+1$. We have
         $\init(f) = z_1 (z_1-z_2)^2$, and any boundary point of the 
         complement of $\mathcal{I}(\init(f))$ satisfies $z_1=0$ or $z_1=z_2$. 
         Altogether, $\mathcal{I}(f)^{\compl}$ has six components.}
	\label{fig:ex-convexity-neu}
\end{figure}

Hence, there are at least as many unbounded components in $\mathcal{I}(f)^{\compl}$ as components in $\mathcal{I}(\init(f))^{\compl}$. Moreover, if $\init(f)$ is hyperbolic, 
$\mathcal{I}(f)^{\compl}$ has at least two (full-dimensional) components.
Note that for a polynomial $f$, the terms of lower degree can cause some unbounded 
components in the complement that have lower-dimensional recession cones.
See Figure \ref{fig:ex-convexity-neu} for an example.

In order to prove the theorem, we show the following lemma, where $\inter$ denotes the
interior of a set.

\begin{lemma}\label{lemma:limit-directions-LT-f}
For $f \in \C[\mathbf{z}]$, the following statements hold.
\begin{enumerate}
\item The sets of limit directions $\mathcal{I}_\infty(f)$ and $\mathcal{I}_\infty(\init(f))$ coincide.
\item If $C$ is an unbounded component of $(\overline{\mathcal{I}(f)})^{\compl}$ with recession cone 
$C'$, then $\inter C'$ is a component of $\mathcal{I}(\init(f))^{\compl}$ if and only if 
$\dim C' = n$.
\item If $C$ is a component of $\mathcal{I}(\init(f))^{\compl}$, then there is a $\mathbf{y}_0\in\R^n$ such that $\mathbf{y}_0+C$ lies in a component of $\mathcal{I}(f)^{\compl}$ and $C$ equals the interior of the
recession cone of that complement component.
\end{enumerate}
\end{lemma}

\begin{proof}
(1)
The homogenization $f_h$ has a zero at infinity, i.e. $(0,z_1,\ldots,z_n)\subseteq\mathcal{V}(f_h)$, if and only if $\init(f)=f_h(0,z_1,\ldots,z_n)=0$. 
Hence, the limit directions of $f$ and $\init(f)$ coincide.

(2) Since hyperbolicity cones are open, $\mathcal{I}(\init(f))$ is closed and thus
    $\mathcal{I}_\infty(f)=\mathcal{I}_\infty(\init (f))$ is closed as well.

Let $C$ be an unbounded component of $(\overline{\mathcal{I}(f)})^{\compl}$ with recession cone $C'$.
If $\dim C' < n$, then $\inter C' = \emptyset$, hence $C'$ is not a hyperbolicity cone of the 
homogeneous polynomial $\init(f)$.
Conversely, if $\dim C = n$ then let $\mathbf{y}_0 \in \R^n$ with $\mathbf{y}_0 + C' \subseteq C$.
For all $r>0$ we have
\[
\frac{1}{r}\big((\mathbf{y}_0+C')\cap\mathcal{I}(f)\big)\cap\Sph^{n-1}=\emptyset.
\]
Under taking the limit $r\rightarrow\infty$, we obtain that no interior point of the
set of limit points
\begin{equation}\label{eq:limit-direction-underdimensional-complement-components}
\lim_{r\rightarrow\infty}\frac{1}{r}(\mathbf{y}_0+C')\cap\Sph^{n-1}
\end{equation}
is a limit direction of $\mathcal{I}(f)$. 
By (1), these interior points are not limit directions of $\mathcal{I}(\init (f))$ either.
As a consequence, $\inter C'$ is a component
of $\mathcal{I}(\init(f))^{\compl}$.

(3) Let $C'$ be a component of $\mathcal{I}(\init (f))^{\compl}$. Set $U=C'\cap\Sph^{n-1}$ and note that the positive hull $\pos U$ satisfies 
$\pos U=C'$. Since $\mathcal{I}(\init (f))$ is a cone, we have $U\subseteq\mathcal{I}_\infty(\init (f))
= \mathcal{I}_\infty(f)$. Hence, there is a $\mathbf{y}_0\in\mathcal{I}(f)^{\compl}$ such that $\mathbf{y}_0+\pos U$ is contained in a component $\mathcal{I}(f)^{\compl}$. 

Denote by $C''$ the recession cone of the component of $\mathcal{I}(f)^{\compl}$ that contains $\mathbf{y}_0+C'$. Clearly, $C'\subseteq C''$. Using (2), it follows that $\inter C''= C'$.
\end{proof}

Theorem~\ref{thm:unbounded-complement-components-dim-n} is a consequence of Lemma \ref{lemma:limit-directions-LT-f}.

\medskip

\noindent
\emph{Proof of Theorem~\ref{thm:unbounded-complement-components-dim-n}.}
  If the recession cone $C'$ of $C$ is full-dimensional, then, by
Lemma \ref{lemma:limit-directions-LT-f}~(2), $\inter C'$ is a
component of $\mathcal{I}(\init(f))^{\compl}$,
i.e., $\inter C'$ is a hyperbolicity cone of $\init(f)$.

Conversely, if the recession cone $C'$ of $C$ is a hyperbolicity cone of $\init(f)$,
then, by Lemma~\ref{lemma:limit-directions-LT-f}~(3), it is open
and thus full-dimensional.
\hfill $\Box$

\medskip

We now show Theorem~\ref{th:strict-convex-nonhomog}.
For $\varphi \in \R$, denote by
$R^{\varphi} : \R^2 \to \R^2$ the linear mapping rotating a given point $\mathbf{x} \in \R^2$ 
by an angle $\varphi$ around the origin. $R^{\varphi}$ has a real representation matrix and
can also be viewed as a linear mapping $\C^2 \to \C^2$.

\begin{proof}[Proof of Theorem \ref{th:strict-convex-nonhomog}]
      Given $K \in \N$, we construct a polynomial $p_{K,n}$ in $n$ variables with at least
	$K$ strictly convex complement components. For the case $n=2$, let
	\[
	g(z_1,z_2) \ = \ (-z_1^2+z_2^2-1)(z_1^2-z_2^2-1) \, ,
	\]
      and
	\begin{equation}\label{eq:poylnomial-strictly-convex-comp-comp}
	p_{K,2}(\mathbf{z}) \ = \ \left(z_1^2+ z_2^2+r^2\right)\cdot \prod_{j=0}^{m-1}g(R_1^{2\pi j/m}(z_1,z_2),R_2^{2\pi j/m}(z_1,z_2))
	\end{equation}
	where $m=\left\lceil\frac{K}{4}\right\rceil$ and $r>0$ sufficiently large. By \cite[Thm. 5.3]{TTT}, 
$\mathcal{I}(z_1^2+z_2^2+r^2)^{\compl}$ is the open disk with radius $r$ centered at the origin, and the boundaries of the two-dimensional
components of $\mathcal{I}(g)^{\compl}$ are given by four hyperbolas.
Since the convex components of $\mathcal{I}(g)^{\compl}$ and of 
$\mathcal{I}(z_1^2+z_2^2+r^2)^{\compl}$ are strictly convex, the components
of $\mathcal{I}(p_{K,2})^{\compl}$ are strictly convex. Figure~\ref{im:p_2} depicts $\mathcal{I}(p_{4,2})$.
	
	\ifpictures
	\begin{figure}[ht]
		\[
		\includegraphics[height=0.25\linewidth]{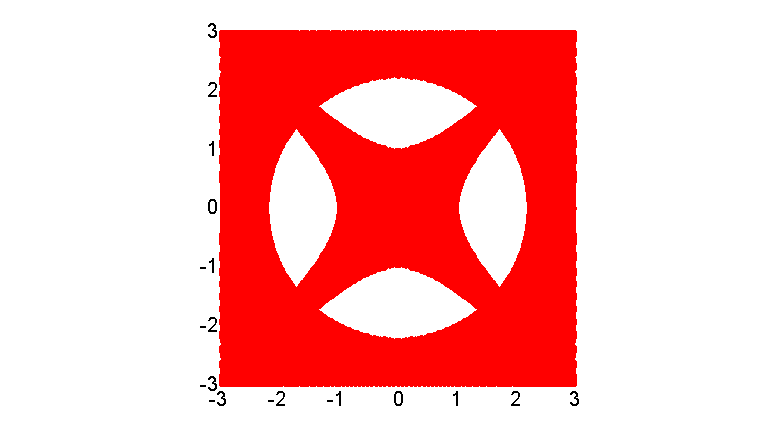}
		\]
		\caption{The imaginary projections of $p_{4,2}$.}
		\label{im:p_2}
	\end{figure}
	\fi
	
	The expressions $(R^{2\pi j/m}_1(z_1,z_2),R^{2\pi j/m}_2(z_1,z_2))$ in the arguments of $g$ provide a rotation of its imaginary projection by an angle of $-2\pi j/m$. Choosing $r$ large enough guarantees that the complement component of $\mathcal{I}(z_1^2+z_2^2+r^2)$ is not completely covered by the imaginary projections of the $g(R_1^{2\pi j/m}(z_1,z_2),R_2^{2\pi j/m}(z_1,z_2))$.
Altogether, $\mathcal{I}(p_{K,2})^{\compl}$ has $4m\geq K$ bounded and strictly convex,
two-dimensional components.
	
	Note that the asymptotes of the hyperbolas do not belong to the imaginary projection of $g$, except the origin. Therefore, $\mathcal{I}(p_{K,2})^{\compl}$ has in total $8m$ bounded components.

The case $n \ge 3$ follows by a suitable modification 
of \eqref{eq:poylnomial-strictly-convex-comp-comp}. Namely, set
      \[
        g(\mathbf{z}) \ = \ (r z_1)^2 - \big( \sum_{j=2}^n z_j^2 \big) + 1 \ = \ r^2 z_1^2 - \big( \sum_{j=2}^n z_j^2 \big) + 1
      \]
      and
      \[
	p_{K,n}(\mathbf{z}) \ = \left(\sum_{j=1}^n z_j^2+1\right) \cdot 
       \prod_{j=0}^{m-1}g(R^{2 \pi j /m}_1(z_1,z_2),R^{2 \pi j /m}_2(z_1,z_2), z_3, \ldots, z_n) \, ,
	\]
      where $m=\left\lceil\frac{K}{2}\right\rceil$.
       $\mathcal{I}(\sum_{j=1}^n z_j^2+1)^{\compl}$ is the open ball in $\R^n$ with radius $1$ centered 
      at the origin, and by \cite[Thm.\ 5.4]{TTT} the boundaries of the two convex components of 
      $\mathcal{I}(g)^{\compl}$ are given by 
      $B_1 := \{\mathbf{y} \in \R^n \, : \, y_1 \ge  1/r \text{ and } r^2 y_1^2 - \sum_{j=2}^n y_j^2  = 1\}$
      and
      $B_2 := \{\mathbf{y} \in \R^n \, : \, y_1 \le -1/r \text{ and } r^2 y_1^2 - \sum_{j=2}^n y_j^2  = 1\}$.
Since $B_1 \cup B_2$ is a two-sheeted $n$-dimensional hyperboloid, $B_1$
and $B_2$ are the boundaries of strictly convex sets.
Note that for $r \to \infty$, the set 
      $\mathcal{I}(g)$ converges to the $y_1$-hyperplane
      on all compact regions of $\R^n$.
        
      	Again, since the rotation $(R_1^{2 \pi j /m}(z_1,z_2),R_2^{2 \pi j / m}(z_1,z_2))$ in the arguments of $g$ induce a rotation of its 
      imaginary projection by an angle of $-2\pi j/m$ with respect to the $y_1y_2$-plane, 
      choosing $r$ large enough gives $2m \geq K$ bounded and strictly convex components.
\end{proof}

\section{Conclusion and open question}

We have provided quantitative and convex-geometric results on the complement
components of imaginary projections and of the hyperbolicity cones of
hyperbolic polynomials.
In the case of amoebas of polynomials, 
to every complement component an order can be associated
(see \cite{fpt-2000} for this order map).
In the homogeneous case of imaginary projections, the direction vectors of the 
hyperbolicity cones can be regarded as a (non unique) representative of an order 
map. And for the unbounded complement components
of non-homogeneous polynomials, Theorem~\ref{thm:unbounded-complement-components-dim-n}
establishes a connection via the initial form.
It is an open question, whether a variant or generalization
of this also holds for the bounded complement components in case of non-homogeneous
polynomials.

Moreover, Shamovich and Vinnikov \cite{shamovich-vinnikov-2014}
recently studied generalizations of hyperbolic polynomials in terms of
hyperbolic varieties, and it would be interesting to extend our results
to that setting.

\section*{Acknowledgments}

We thank Mario Kummer and Pedro Lauridsen Ribeiro for comments and corrections
on an earlier version.

\bibliographystyle{plain}

\end{document}